\documentclass[11 pt]{amsart}
\usepackage{srcltx}
\usepackage{eurosym}
\usepackage{mathtools}
\usepackage{amsmath}
\usepackage{amsfonts}
\usepackage{amssymb}
\usepackage{amsthm}
\usepackage{graphicx}
\usepackage{mathrsfs}
\usepackage{xcolor}
\usepackage{exscale}
\usepackage{latexsym}
\usepackage{dirtytalk}
\usepackage{graphicx}
\usepackage[normalem]{ulem}
\usepackage{fancyhdr}
\usepackage[margin=3cm]{geometry}
\usepackage{hyperref}
\hypersetup{
 colorlinks   = true,
 urlcolor     = blue,
 linkcolor    = blue,
 citecolor   = red ,
 bookmarksopen=true
}

\def\YYint#1#2#3{{\setbox0=\hbox{$#1{#2#3}{\iint}$}
    \vcenter{\hbox{$#2#3$}}\kern-.50\wd0}}

\def\XXint#1#2#3{{\setbox0=\hbox{$#1{#2#3}{\int}$}
    \vcenter{\hbox{$#2#3$}}\kern-.50\wd0}}

\makeatletter
\def\namedlabel#1#2{\begingroup
   \def\@currentlabel{#2}%
   \label{#1}\endgroup
}
\makeatother
\makeatletter
\newcommand{\rmh}[1]{\mathpalette{\raisem@th{#1}}}
\newcommand{\raisem@th}[3]{\hspace*{-1pt}\raisebox{#1}{$#2#3$}}
\makeatother

\newtheorem{thm}{Theorem}[section]

\newtheorem{lem}[thm]{Lemma}

\theoremstyle{definition}
\theoremstyle{Remark}
\newtheorem{defn}[thm]{Definition}
\newtheorem{rem}{Remark}
\theoremstyle{example}

\numberwithin{equation}{section}

\newcommand{\R}{\mathbb{R}}

\newcommand{\ve}{\varepsilon}
\newcommand{\I}{\mathbb{I}}

\begin{document}
\title[]
 {Gradient continuity estimates  for the  normalized $p-$poisson equation}

 \author{Agnid Banerjee}
\address{Tata Institute of Fundamental Research\\
Centre For Applicable Mathematics \\ Bangalore-560065, India}\email[Agnid Banerjee]{agnidban@gmail.com}

\author{Isidro H. Munive}
\address{Instituto de Matem\'aticas, M\'exico}\email[Isidro Munive]{imunivel@gmail.com}

 \thanks{First author is supported in part by SERB Matrix grant MTR/2018/000267}
 \thanks{Second author is supported by CONACYT grant 265667, Instituto de Matem\'aticas, UNAM}


\begin{abstract}
In this paper, we obtain   gradient continuity estimates for viscosity  solutions of  $\Delta_{p}^N u= f$ in terms of the  scaling critical $L(n,1 )$  norm of $f$,   where $\Delta_{p}^N$ is the normalized $p-$Laplacian operator defined in \eqref{pl} below.  Our main result, Theorem \ref{main},  corresponds to the borderline gradient  continuity estimate in terms of the modified Riesz potential $\tilde \I^{f}_{q}$. Moreover, for $f \in L^{m}$ with $m>n$, we also obtain $C^{1,\alpha}$ estimates,  see Theorem \ref{main1} below. This  improves  one of the   regularity results   in  \cite{APR},  where  a  $C^{1,\alpha}$  estimate was  established  depending on the  $L^{m}$ norm of $f$ under the additional restriction  that  $p>2$ and $m > \text{max} (2,n, \frac{p}{2}) $ (see Theorem 1.2 in \cite{APR}).  We also  mention that differently from the approach in \cite{APR}, which uses methods from divergence form theory and nonlinear potential theory in the proof of Theorem 1.2, our method  is more non-variational in nature, and it is   based on separation of phases inspired by the  ideas in \cite{W}.  Moreover, for $f$ continuous, our approach also  gives  a somewhat different  proof of the $C^{1, \alpha}$ regularity result, Theorem 1.1,  in \cite{APR}. 
\end{abstract}

\maketitle

\tableofcontents 

\section{Introduction}
The aim of this paper is to   obtain pointwise gradient continuity estimates   for viscosity solutions of
\begin{equation}
\label{m}
\Delta_{p}^N u= f
\end{equation}
in terms of the scaling critical $L(n,1)-$norm of $f$. Here, $\Delta_{p}^N $ denotes the normalized $p-$Laplace operator given by
\begin{equation}
\label{pl}
\Delta_{p}^N u\doteq \bigg(\delta_{ij}+ (p-2) \frac{u_i u_j}{ |\nabla u|^2} \bigg) u_{ij}.
\end{equation}

The fundamental role of these borderline, or end-point regularity, estimates in the theory of elliptic and parabolic partial differential equations is well known.  In order to put our result in the correct historical perspective, we note that in 1981,  E. Stein in his visionary work \cite{MR607898} showed the following.
\begin{thm}\label{stein}
Let $L(n,1)$ denote the standard Lorentz space, then the following implication holds:
\[\nabla v \in L(n,1) \ \implies \ v\  \text{\emph{is continuous}}.\]  
\end{thm}
The Lorentz space $L(n,1)$ appearing in Theorem \ref{stein}  consists of those measurable functions $g$ satisfying the condition
\[
\int_{0}^{\infty} |\{x: g(x) > t\}|^{1/n} dt < \infty.
\]
Theorem \ref{stein} can be regarded as the limiting case of Sobolev-Morrey embedding that asserts
\[
\nabla v \in L^{n+\ve} \implies v \in C^{0, \frac{\ve}{n+\ve}}.
\]
Note that indeed $L^{n+\ve} \subset L(n, 1) \subset L^{n}$ for any $\ve>0$, with all the inclusions being strict.  Now Theorem \ref{stein} coupled with the standard Calderon-Zygmund theory  has the following interesting consequence.
\begin{thm}\label{st}
$\Delta u \in L(n,1) \implies \nabla u$ is continuous.
\end{thm} 
The analogue of Theorem \ref{st}  for general nonlinear, and possibly degenerate elliptic and parabolic equations, has become accessible not so long ago through a rather sophisticated and powerful nonlinear potential theory (see for instance \cite{MR2823872,MR2900466,MR3174278} and the references therein).  The first breakthrough in this direction came up in the work  of Kuusi and Mingione in \cite{MR3004772}, where they showed that the analogue of Theorem \ref{st}  holds for operators modelled after the $p$-Laplacian. Such a result  was subsequently generalized to $p$-Laplacian-type systems by the same authors in  \cite{MR3247381}. 

Since then, there has been several generalizations of Theorem \ref{st} to operators with various kinds of nonlinearities. In the context of  fully nonlinear elliptic equations, the  analogue of Theorem \ref{st} was established by Daskalopoulos-Kuusi-Mingione in  
\cite{DKM}. More precisely, they showed the following (see Theorem 1.1 in \cite{DKM}).
\begin{thm}\label{dkm}
Let $u$ be a  $W^{2,q}$ viscosity solution of
\begin{equation}\label{fl}
F(x,\nabla^2u)= f\   \text{in}\ B_1,
\end{equation}
where $F$ is uniformly elliptic fully nonlinear operator and $f \in L(n, 1)$. Then, there exists $\theta \in (0, 1)$, depending only on $n$ and the ellipticity constants of $F$,  such that if $F(.)$ has $\theta$-BMO coefficients, then $\nabla u$ is continuous in the interior of $B_1$. Moreover, the following estimates hold for some $\alpha=\alpha(n, \lambda, \Lambda)$ and $\delta \in (0,1)$, 
\begin{equation}\label{dkm1}
\begin{cases}
|\nabla u(x_0)| \leq  C\left( \tilde \I^{f}_{q} (x_0,r) + \left(\frac{1}{|B_{r} (x_0)|}\int_{B_{r}(x_0)} |\nabla u|^{p}\right)^{1/p}\right)\ \text{\emph{for any} $p>n$},
\\
|\nabla  u(x_1) - \nabla u(x_2) | \leq C \left( ||\nabla u||_{L^{\infty}(B_{3r})} |x_1- x_2|^{\alpha(1-\delta)} +  \sup_{x \in \{x_1, x_2\}}  \tilde \I^{f}_{q} (x, 4 |x_1-x_2|^{\delta})\right),
\end{cases}
\end{equation}
whenever $x_1, x_2 \in B_{r}$. Here  $\tilde \I^{f}_{q} (x_0,r )$ is the  \say{modified Riesz potential} defined by
\begin{equation}\label{riesz}
\tilde \I^{f}_{q} (x_0,r )= \int_{0}^{r} \left( \def\avint{\mathop{\,\rlap{-}\!\!\int}\nolimits} \avint_{B_s(x_0)} |f|^q \right)^{1/q} ds,
\end{equation} 
and $C=C(n,p,r,\lambda,\Lambda)$.
\end{thm}
Before proceeding further, we make the following important remark. 

\begin{rem}\label{mod}
The reader should note that  from the Hardy-Littlewood rearrangement inequality (see for instance \cite{DKM}) we have that 
\begin{align}\label{int1}
&  \int_{0}^{r} \left( \def\avint{\mathop{\,\rlap{-}\!\!\int}\nolimits} \avint_{B_s} |f|^q\right)^{1/q} ds \leq   \frac{C}{|B_1|^{\frac1n}} \int_{0}^{|B_r|} \left[ f^{**}(\rho) \rho^{\frac{q}{n}} \right]^{\frac{1}{q}} \ \frac{d\rho}{\rho},
\end{align}
where $f^{**}$ is defined as 
\[
f^{**}(\rho)= \frac{1}{\rho} \int_{0}^{\rho}  f^{*} (t) dt,
\]
with $f^{*}$ being the radial non-increasing rearrangement of $f$.  Now, when $f \in L(n,1)$, we have from an equivalent characterization of Lorentz spaces that
 \begin{equation}\label{t}
   \int_{0}^{\infty} \left[ f^{**}(\rho) \rho^{\frac{q}{n}} \right]^{\frac{1}{q}} \ \frac{d\rho}{\rho} < \infty, \quad \text{for $q < n$}.
   \end{equation}
Therefore, it follows from the inequalities in \eqref{int1} and \eqref{t}   that when $f \in L(n,1)$ and $q<n$,   $\tilde \I^{f}_{q} (x_0,r ) \to 0$ as $r \to 0$.  Consequently, the gradient continuity follows from the estimates in \eqref{dkm1} above.  
\end{rem}
We also refer to the recent work \cite{AB} of one of us and Adimurthi where  an analogous regularity  result has been obtained  under Dirichlet boundary conditions when the domain is  $C^{1, Dini}$. The result was established  using Caffarelli style compactness arguments as in \cite{Ca}. 

\medskip 

In this paper we establish a similar estimate  as in \eqref{dkm1} above  when the fully nonlinear operator $F$ gets replaced by the  normalized $p-$Laplacian operator $\Delta_{p}^N$.  In order to provide the reader  with the right viewpoint concerning our approach, we   note that getting $C^{1}-$regularity result  in general   amounts to show that the graph of $u$ can be touched by an affine function so that the error is of order $o(r)$ in a ball of radius $r$ for every $r$ small enough. The proof of this is based on iterative argument where one ensures improvement of flatness at  every successive scale by comparing to a solution of a  limiting equation with  more regularity.  At each step, via rescaling, it reduces to show that if $<p_0,x> +u$ solves \eqref{m} in $B_1$, then the oscillation of $u$ is strictly smaller in a smaller ball upto a linear function. This is accomplished via compactness arguments which crucially relies on apriori estimates. Such estimates in the context of $\Delta_{p}^N$ come from the Krylov-Safonov theory because the equation \eqref{m} lends itself to a uniformly elliptic structure.  

\medskip

Now, for a $u$ that solves \eqref{m}, we have that $u-<p_0,x>$ is a solution of the following perturbed equation
\begin{equation}\label{it}
\bigg(\delta_{ij} + (p-2) \frac{(u_i +(p_0)_i)(u_j+(p_0)_j)}{|\nabla u + p_0|^2} \bigg) u_{ij}=f.
\end{equation}
Therefore, in order to obtain improvement of flatness at each scale after a rescaling, it is imperative to get uniform $C^{1}-$type estimates independent of $|p_0|$  for  the limiting equations corresponding to the case  $f \equiv 0$.  This is precisely done in \cite{APR}  by an adaptation of  the Ishii-Lions approach as in \cite{IL},  where the authors   obtained uniform Lipschitz estimates  for solutions to \eqref{it} for large $|p_0|'s$ when $f=0$.  In this paper, we follow an approach which is different from that in \cite{APR}.  Our proofs of Theorem \ref{main} and Theorem \ref{main1}  are based  instead on separation of  the degenerate and the non-degenerate phase,  and  do not rely on the uniform Lipschitz estimates for equations of the type \eqref{it} for large $|p_0|'s$.  This is inspired by ideas in \cite{W}, where an alternate proof of $C^{1, \alpha}-$regularity for the $p-$Laplacian was given.  Moreover, in the case of continuous $f$,  our method  also provides a different  proof of the $C^{1, \alpha}-$regularity result for \eqref{m}  established  in \cite{APR} (see also \cite{BD} for $p\geq 2$). We believe that this alternate viewpoint    would definitely   be of independent interest. 

\medskip

Finally, we mention that over the last decade, there has been a growing  attention on equations of the type \eqref{m} because of  their connections to tug-of-war games with noise. This aspect was first studied in \cite{PS} .  In recent times,   the parabolic  normalized $p-$Laplacian,  as well as its degenerate and singular variants, have been studied in  various contexts in  a number of papers, see \cite{A, JK, D, BG1, BG2, BG3, HL,  PR, IJS, Ju, JS, MPR}.  Such equations have also found applications in image processing (see for instance \cite{D}). 

\medskip

The paper is organized as follows. In Section \ref{n} we introduce some basic notations, list some preliminary results, and then state our main theorems.  In Section \ref{mn} we first establish   approximation lemmas that play a crucial role in the separation of phases in the  iterative argument in the proof  of our main results. We then subsequently establish our  main results Theorem \ref{main} and Theorem \ref{main1}.  In closing, we would like to mention that it  remains to be seen whether one can obtain similar borderline estimates for more general  equations of the type
\[
|\nabla u|^\gamma \bigg(\delta_{ij} + (p-2) \frac{u_i u_j}{|\nabla u|^2} \bigg) u_{ij} =f,
\]
with appropriate restrictions on the parameter $\gamma$. This seems to be an interesting open question to which we would like to come back in a future study.

\section{Notations, Preliminaries and  statement of the main results}\label{n}

We  denote  points in $\R^n$  by $x,y,x_1,x_2$ etc. We let $|x|$ be the norm of $x$, and $|A|$ will denote the Lebesgue measure of $A\subset \R^n$. Let $B_r(x)=\{x:|x|<r\}$. When $x=0$, we will ocassionally denote such a set by $B_r$. By $\partial B_r(x)$, we will denote the boundary of the set  $B_{r}(x)$.   We will also denote by $S(n)$ the space of $n \times n$ symmetric matrices.    In our ensuing discussion, at times we will be using the notation  $\def\avint{\mathop{\,\rlap{-}\!\!\int}\nolimits} \avint_{A} h  dx$  to indicate  the integral average of a function $h$ over a set $A$.

\medskip

We now  fix an exponent $q \in (n-n_0,n)$, where $n_0$ (denoted by $\ve$ in \cite{E}) is a small universal constant as obtained in \cite{E}, such that the Krylov-Safanov type H\"older estimate  holds  for  functions  which belong to extremal Pucci class   $\mathcal{S}(\lambda, \Lambda, f)$  in the $W^{2,q}$ viscosity sense. Here  
\begin{equation}\label{l}
\lambda=\text{min}(1, p-1),\  \Lambda= \text{max}(1, p-1),
\end{equation}
$f \in L^{q}$, and   $\mathcal{S}(\lambda, \Lambda, f)$ is the set of all functions $u$ which solves in the $W^{2,q}$ viscosity sense (we refer to \cite{CCCS} for the precise notion of $W^{2.q}$ viscosity solutions)
 \begin{equation}
 \mathcal{P}_{\lambda, \Lambda}^{-} (\nabla^2 u) \leq f \leq  \mathcal{P}_{\lambda, \Lambda}^{+}  (\nabla^2 u).  
 \end{equation}
The operators $\mathcal{P}_{\lambda, \Lambda}^{-}$ and $ \mathcal{P}_{\lambda, \Lambda}^{+}$ are the minimal and maximal Pucci operators, respectively, defined in the following way
 \begin{equation}\label{max}
 \begin{cases}
 \mathcal{P}_{\lambda, \Lambda}^{-}( M)= \text{inf}_{\{A \in S(n): \lambda \mathbb{I} \leq A \leq \Lambda \mathbb{I}\}} \text{trace}\ (AM), 
 \\
 \mathcal{P}_{\lambda, \Lambda}^{+}( M)= \text{sup}_{\{A \in S(n): \lambda \mathbb{I} \leq A \leq \Lambda \mathbb{I}\}} \text{trace}\ (AM).
 \end{cases}
 \end{equation}

We now turn our attention to the relevant notion of solution to \eqref{m}. For $p \in \R^n -\{0\}$ and $X=[m_{ij}] \in  S(n)$, following \cite{BK},  we define 
\[
F(p, X)= \bigg(\delta_{ij} + (p-2) \frac{p_i p_j}{|p|^2} \bigg)  m_{ij}.
\]
Then as in \cite{CIL},  the lower semicontinuous relaxation $F_*$ is defined as follows

\begin{equation}
F_*(q,X)=\begin{cases}  \qquad F(q,X)\quad&\hbox{if }q\not=0,\\
\inf_{a\in{\mathbb R}^n\setminus\{ 0\}}F(a,X)\quad&\hbox{if }q=0,      
\end{cases}\end{equation}
while the upper semicontinuous relaxation $F^{*}$ is defined as 

\begin{equation}\label{up}
F^{*}(q,X)=\begin{cases} \qquad F(q,X)\quad&\hbox{if }q\not=0,\\
\sup_{a\in{\mathbb R}^n\setminus\{ 0\}}F(a,X)\quad&\hbox{if }q=0.     
\end{cases}
\end{equation}
\begin{defn}
We say that $u$ is a $W^{2,q}$ viscosity sub-solution of \eqref{m} in a domain $\Omega \subset \R^n$ if given  $\phi \in W^{2,q}$ such that $u-\phi$ has a local maximum at  $x_0 \in \Omega$, then one has
\begin{equation}\label{test}
\limsup_{x \to x_0}  F^{*} ( \nabla \phi(x), \nabla^2 \phi(x)) - f(x) \geq 0.
\end{equation}
\end{defn}
In an analogous way, the notion of viscosity supersolution of \eqref{m} is defined  using $F_*$ instead of $F^{*}$, and where $\limsup$ gets replaced by $\liminf$ in the equation \eqref{test} above.  Finally, we  say that $u$ is a $W^{2,q}$ viscosity solution to \eqref{m} if it is both a subsolution and a supersolution.  It is easy to  deduce that if $u$ is a $W^{2,q}$ viscosity solution to \eqref{m}, then $u$ belongs to the Pucci class $\mathcal{S}(\lambda, \Lambda, f)$ in the $W^{2,q}$ viscosity sense where $\lambda, \Lambda$ are as in \eqref{l}. Hence, $u$ satisfies  universal H\"older estimates as in \cite{E}, which  depend on $n, p$ and $||u||_{L^{\infty}}$.

\subsection{Statement of the main results}
We now state our first  main result. This result corresponds to the  regularity  estimate  in the borderline  case, i.e.,  gradient continuity estimates with dependence on the  $L(n,1)$ norm of $f$. 

\begin{thm}\label{main}
For a given $p>1$, let $u$  be a $W^{2,q}$ viscosity solution of \eqref{m} in $B_1$  where $f \in L(n,1)$. Then $\nabla u$ is continuous inside of $B_1$. Moreover, the following  borderline estimates hold
\begin{equation}\label{bm}
\begin{cases}
&|\nabla u(x_0)| \leq  C( \tilde \I^{f}_{q} (x_0,1/2 ) + ||u||_{L^{\infty}(B_1)})\ \text{ \emph{for} $x_0 \in B_{1/2}$},
\\
&|\nabla  u(x_1) - \nabla u(x_2) |\\ 
& \leq C(n, p) \bigg( \bigg[|| u||_{L^{\infty}(B_{3/4})}  + \sup_{x \in \{x_1, x_2\}}  \tilde \I^{f}_{q} (x, 1) \bigg] |x_1-x_2|^{\alpha/4} +  \sup_{x \in \{x_1, x_2\}}  \tilde \I^{f}_{q} (x, 4 |x_1-x_2|^{1/4}) \bigg),
\end{cases}
\end{equation}
  whenever  $x_1, x_2 \in B_{1/2}$,  and  where  $\alpha=\alpha(n, p)$.  \end{thm}
  
  In the case  $f \in L^{m}(\R^n)$ with   $m>n$, we  obtain the following regularity result that improves Theorem 1.2 in \cite{APR}.
  
  \begin{thm}\label{main1}
 For $p>1$ and $m>n$, let $u$ be a $W^{2,m}$ viscosity solution of \eqref{m} in $B_1$, where $f \in L^{m}$. Then, $\nabla u \in C^{\alpha_0}(\overline{B_{1/2}})$ for some $\alpha_0=\alpha_0(n, p, m)$. Moreover, we have that the  following estimate holds,
 \[
 || u||_{C^{1,\alpha_0}(B_{1/2})} \leq C(n, p, ||f||_{L^m}, ||u||_{L^{\infty}(B_1)}).
 \]

  \end{thm}

\section{Proof of the main results}\label{mn}

\subsection{Proof of Theorem \ref{main}}

 We now fix a universal parameter which plays a crucial role in our compactness arguments.  Let $\beta>0$ be the optimal H\"older exponent such that any arbitrary solution $u$   of
\[
\operatorname{div}(|\nabla u|^{p-2} \nabla u)=0\quad \text{is in $C^{1, \beta}_{loc}$}.
\]
The fact  that $\beta>0$ follows from the regularity results in \cite{Db}, \cite{Le} and \cite{To}. We then   fix some  $\alpha>0$ such that 
\begin{equation}\label{universal}
\alpha < \beta.
\end{equation} 

We now state our first relevant approximation lemma which plays a very crucial role in the separation of phases. This is analogous to Lemma 2.3 in \cite{W}.

\begin{lem}\label{app1}
Let $u$ be a  $W^{2, q}$ viscosity solution of
\begin{equation}\label{at}
\bigg(\delta_{ij} + (p-2) \frac{(\delta u_i+A_i)(\delta u_j +A_j)}{ |\delta \nabla u+A|^2} \bigg) u_{ij}=f\quad \text{in $B_1$},
\end{equation}
 with $|u| \leq 1$, $u(0)=0$  and $|A| \geq 1$.  Given $\tau>0$, there exists  $\delta_0=\delta_0(\tau)>0$  such that if 
\[
\delta, \left(\frac{1}{|B_{3/4}|} \int_{B_{3/4}}  |f|^q \right)^{1/q} \leq \delta_0,
\]
then  for some  $w \in C^{2}(\overline{B_{1/2}})$  with universal $C^{2}$ bounds depending only on $n, p$ and  independent of $|A|$, we have that
\begin{equation}
\begin{cases}
w(0)=0
\\
||w-u||_{L^{\infty}(B_{1/2})} \leq \tau
\end{cases}
\end{equation}


\end{lem}

\begin{proof}
We argue by contradiction. If not, then there exists $\tau_0>0$ and   a sequence of pairs  $\{u_k, f_k \}$ that solves \eqref{at} corresponding to $\{\delta_k, A_k\}$ with  $\delta_k \to 0, f_k \to 0\ \text{in $L^q(B_{3/4})$}$ as $k \to \infty$ and  such  that $u_k's$ are not $\tau_0$ close to any such $w$.  We note that the equation satisfied by $u_k$ can be rewritten as 
\begin{equation}\label{at1}
\bigg(\delta_{ij} +(p-2) \frac{(\tilde \delta_k (u_k)_i + (\tilde A_k)_i ) (\tilde \delta_k (u_k)_j + (\tilde A_k)_j)}{ |\tilde \delta_k \nabla u_k + \tilde A_k|^2}  \bigg) (u_k)_{ij} =f_k,
\end{equation}
where $\tilde \delta_k = \frac{\delta_k}{|A_k|}$ and $\tilde A_k= \frac{A_k}{|A_k|}$. Since $|A_k| \geq 1$, we have $\tilde \delta_k \to 0$ as $k \to \infty$.

From the Krylov-Safonov-type estimates as in \cite{E}, we now observe that  $u_k$'s are uniformly H\"older continuous in $\overline{B_{3/4}}$. Therefore, upto a subsequence,  by Arzela-Ascoli we may assume that $u_k \to u_0$ uniformly on $B_{3/4}$ and, moreover, we can also assume that  $\tilde A_k  \to A_0$(by possibly passing to   another subsequence)  such that  $|A_0|=1$. 

We now make the following claim.

\emph{Claim:} $u_0$ solves 
\begin{equation}\label{lim}
\bigg(\delta_{ij} + (p-2) (A_0)_i (A_0)_j \bigg) (u_0)_{ij} =0.
\end{equation}
By standard theory, it suffices to check that $u_0$ is a $C^{2}-$viscosity solution to the above limiting equation.  We note that the stability result in Theorem 3.8 in \cite{CCCS} can not be directly applied here,  because of the singular dependence of the operator in the \say{gradient} variable. We, however, show that the proof of Theorem 3.8 can still be adapted in this situation.  Let $\phi$ be a $C^{2}$ function such that the graph of  $\phi$ strictly touches the graph of $u_0$ from above  at $x_0 \in B_{1/2}$. We show that  at $x_0$, 
\begin{equation}\label{et}
\bigg(\delta_{ij} + (p-2) (A_0)_i (A_0)_j \bigg) \phi_{ij} \geq 0.
\end{equation}
Suppose that  is not the case.  Then, there exists $\ve, \eta, r>0$  small enough such that
\begin{equation}\label{con2}
\begin{cases}
\bigg(\delta_{ij} + (p-2) (A_0)_i (A_0)_j \bigg) \phi_{ij} \leq -\ve\ \text{in $B_r(x_0)$},
\\
\phi - u \geq \eta\ \text{on $\partial B_r(x_0)$}.
\end{cases}
\end{equation}
We now show that for every $k$, there exists   perturbed test functions $\phi + \phi_k$ with $\phi_k \in W^{2,q}$ such  that
\begin{equation}\label{con5}
F_k^{*} ( \nabla (\phi +\phi_k), \nabla^{2} (\phi + \phi_k)) \leq f_k - \ve\ \text{in $B_r(x_0)$},
\end{equation}
where $F_k^{*}$ is the upper semicontinuous relaxation of the operator  in \eqref{at1}. Moreover, we can also ensure that  $ (\phi+ \phi_k) - u_k$ has a minimum  in $B_{r}(x_0)$ for large enough $k's$. This would  then contradict the viscosity formulation for $u_k$ for such $k's$ and hence \eqref{et} would follow. 

Therefore, under the assumption that \eqref{con2} holds, we now show the validity of \eqref{con5}.  We first observe that from \eqref{con2}, the following differential inequality holds,
\begin{align}
F_k^{*} ( \nabla (\phi +\phi_k), \nabla^{2} (\phi + \phi_k))\leq & \mathcal{P}_{\lambda, \Lambda}^{+} (\nabla^2 \phi_k) +C_0 |\tilde A_k - A_0|\\  &+ C_0 \tilde \delta_k |\nabla \phi_k| +  C_0 \tilde \delta_k |\nabla \phi| - \ve,
\notag
\end{align}
where $C_0=C_0(||\nabla^2 \phi||, p, n)$ and $\lambda, \Lambda$ are as in \eqref{l}. This  inequality above follows by  adding and subtracting $\bigg(\delta_{ij} + (p-2) (A_0)_i (A_0)_j \bigg) \phi_{ij}$, by using \eqref{con2}, and then by splitting the considerations depending on whether 
\[
|A_0 - (\tilde A_k + \tilde \delta_k (\nabla \phi + \nabla \phi_k))| < 1/2 \ \text{or}\ > 1/2.
\]

 We now let $\phi_k$ be  a strong solution to the following boundary value problem
\begin{equation}
\begin{cases}
\mathcal{P}_{\lambda, \Lambda}^{+} (\nabla^2 \phi_k) + C_0 |\tilde A_k - A_0|  + C_0 \tilde \delta_k |\nabla \phi_k| +  C_0 \tilde \delta_k |\nabla \phi| = f_k\ \text{in $B_r(x_0)$},
\\
\phi_k=0\ \text{on $\partial B_r(x_0)$}.
\end{cases}
\end{equation}
The existence of such strong $W^{2,q}$ solutions is guaranteed by Corollary 3.10 in \cite{CCCS}. Therefore, with such $\phi_k$, we  have that \eqref{con5} holds.  

We now  observe that since  $f_k \to 0$ in $L^{q}$ and also $\tilde \delta_k , |\tilde A_k - A_0| \to 0$, from the generalized maximum principle,  as in \cite{CCCS},  we have  that 
\[
||\phi_k||_{L^{\infty}(B_r)} \to 0\ \text{as $k \to \infty$}.
\]
Now, since $\phi - u$ has a strict  minimum at $x_0$,  it follows for large $k's$ that  $ (\phi+ \phi_k) - u_k$ has  a  minimum in the  inside of $B_r(x_0)$(since $\phi_k \equiv 0$ on $\partial B_r(x_0)$ and $\phi-u > \eta$ on $\partial B_r(x_0)$). From this, as we mentioned before, \eqref{et} follows. 

Then, by an analogous argument we would have that  the opposite inequality holds in \eqref{et}, when  instead the graph of $\phi$  touches the graph of  $u$ from below at $x_0$ and consequently it  follows that $u_0$ solves \eqref{lim}. Moreover, since $|u_0| \leq 1$,  we have from the classical theory  that $u_0$  is smooth with universal $C^{2}$ bounds in $B_{1/2}$.  This would then be a  contradiction for large enough $k'$s since $u_k \to u_0$ uniformly.  This finishes the proof of the lemma.

\end{proof}

As a  consequence of Lemma \ref{app1}, we have the following result on the  affine approximation of $u$ at $0$,  provided there is a sufficiently large non-degenerate slope  at a certain scale.   As the reader will see, such is ensured by the fast geometric convergence of the approximations.

\begin{lem}\label{ap2}
Let $u$ be  a viscosity  solution of
\[
\bigg(\delta_{ij}+ (p-2) \frac{u_i u_j}{|\nabla u|^2} \bigg) u_{ij}=f\quad \text{in $B_1$},
\]
with  $u(0)=0$. Then, there exists a universal $\delta_0>0$  such that if for some  $A \in \R^n$, satisfying  $M \geq|A| \geq 2$,  we have 
\[
||u- <A,x>||_{L^{\infty}(B_1)} \leq \delta_0,
\]
and also 
\[
\int_{0}^{1}  \bigg( \def\avint{\mathop{\,\rlap{-}\!\!\int}\nolimits} \avint_{B_s} |f|^q\bigg)^{1/q} ds \leq \delta_0^2,
\]
then there exists an affine function  $ L_0$  such that
\begin{equation}\label{difr}
\begin{cases}
1 \leq |\nabla L_0| \leq M+1
\\
|u(x)-  L_0(x) | \leq C|x| K(|x|)
\end{cases}
\end{equation}
Here $K(r)\doteq r^{\alpha/2} +  \int_{0}^{r^{1/2}} ( \def\avint{\mathop{\,\rlap{-}\!\!\int}\nolimits} \avint_{B_s} |f|^q)^{1/q} ds$ and $\alpha$ is the universal parameter as in \eqref{universal}. Moreover $\delta_0$ can be chosen independent of $M$.  In view of Remark \ref{mod}, we note that for $f \in L(n, 1)$,  we have that $K(r) \to 0$ as $r \to 0$.  

\end{lem}

\begin{proof}

 We will show that  for  for every $k=0,1,2, \ldots$, there exist linear functions \linebreak $\tilde  L_k x\doteq <A_k, x>$  such that
\begin{equation}\label{cl}
\begin{cases}
||u-\tilde L_k||_{L^{\infty}(B_{r^k})} \leq  r^k \omega(r^k),
\\
|A_k- A_{k-1}| \leq C  \omega(r^{k-1}),
\end{cases}
\end{equation}
for some $r<1$  universal, independent of $\delta_0$.  Here we let for a given $k$,

\begin{equation}\label{om}
\omega(r^k)= \frac{1}{\delta_0} \sum_{i=0}^k r^{i \alpha} \omega_1\left(\frac{3}{4}r^{k-i}\right),
\end{equation}
with  $\omega_1$ defined in the following way
\[
\omega_1(t)=  \max\ \left( t\left( \def\avint{\mathop{\,\rlap{-}\!\!\int}\nolimits} \avint_{B_t
} |f|^q\right)^{1/q}, \delta_0^2 \frac{4}{3} t\right).
\]
We note that $\delta_0$ is to be fixed later.  We also let $A_0\doteq A$.  Now, suppose $A_k$ exists  upto some $k$ with the bounds as in \eqref{cl}. Then, we observe that 
\begin{align}\label{ndeg}
|A_k| &\geq |A_0| -  ( |A_1-  A_0| +\ldots+ |A_k- A_{k-1}|) 
\\
& > 2 -  C  \sum \omega(r^i) > 2 -   \frac{C}{\delta_0} \sum \omega_1 \left(\frac{3}{4}r^i\right)\ \left(\text{using the Cauchy product formula}\right)
\notag
\\
& \geq 2 - C_1 \delta_0 > 1\  \left(\text{if $\delta_0$ is small enough}\right).\notag
\end{align}
In the last inequality above  we also used the fact that
\begin{align}\label{2.0}
 \sum \omega_1(\frac{3}{4} r^i) &\leq C\left( \delta_0^2 \sum r^i +  \sum  \frac{3 r^{i}}{4} \left( \def\avint{\mathop{\,\rlap{-}\!\!\int}\nolimits} \avint_{B_{\frac{3r^i}{4}}} |f|^q\right)^{1/q}  \right)
\\
& \leq C\left( \delta_0^2  +  \int_{0}^{1}  \left( \def\avint{\mathop{\,\rlap{-}\!\!\int}\nolimits} \avint_{B_s} |f|^q\right)^{1/q} ds \right) \leq C_2 \delta_0^2.
\notag
\end{align}
Note that the last inequality in \eqref{2.0} is a consequence of the following  estimate 
\[
\sum  \frac{3 r^{i}}{4} \left( \def\avint{\mathop{\,\rlap{-}\!\!\int}\nolimits} \avint_{B_{\frac{3r^i}{4}}} |f|^q\right)^{1/q}  \leq C \int_{0}^{1}  \left( \def\avint{\mathop{\,\rlap{-}\!\!\int}\nolimits} \avint_{B_s} |f|^q\right)^{1/q} ds,
\]
which in turns  follows by  breaking the integral in the above expression  into integrals  over dyadic  subintervals of the type $[\frac{3}{4} r^{i}, \frac{3}{4} r^{i-1}]$. 

Thus the estimate in \eqref{ndeg}   ensures that the non-degeneracy condition in Lemma \ref{app1} holds for every $k$. We prove the claim in \eqref{cl} by induction.  From the hypothesis of the lemma, the case when $k=0$ is easily verified  with $A_0=A$ with our choice of $\omega$.  Let us now assume that the claim as in \eqref{cl} holds upto some $k$. We then  consider  
\[
v= \frac{u - \tilde  L_k(r^k x)}{ r^k \omega(r^k)},
\]
which solves
\begin{equation}
\bigg(\delta_{ij} + (p-2) \frac{(\omega(r^k) v_i + (A_k)_i)(  \omega(r^k) v_j + (A_k)_j)}{| \omega(r^k) \nabla v+  A_k|^2} \bigg) v_{ij}= \frac{r^k}{ \omega(r^k)} f(r^k x).
\end{equation}
Now, by a change of variable formula and the definition of $\omega$ it follows that,  with
\[
f_k(x)= \frac{r^k}{ \omega(r^k)} f(r^k x),
\]
we have 
\begin{align}\label{com1}
 \bigg(\frac{1}{|B_{3/4}|} \int_{B_{3/4}}  |f_k|^q \bigg)^{1/q} &= \frac{r^k}{ \omega(r^k)} \bigg(\frac{1}{|B_{3r^{k}/4 }|} \int_{B_{\frac{3 r^k}{4}}} |f(y)|^q dy \bigg) ^{1/q}
\\
& \leq \frac{r^k}{\omega_1(\frac{3 r^k}{4}) \frac{1}{\delta_0}} \bigg(\frac{1}{|B_{3r^{k}/4 }|} \int_{B_{\frac{3 r^k}{4}}} |f(y)|^q dy \bigg) ^{1/q}
\notag
\\
& \leq \frac{ r^{k} \bigg(\frac{1}{|B_{3r^{k}/4 }|} \int_{B_{\frac{3 r^k}{4}}} |f(y)|^q dy \bigg) ^{1/q}}{ \frac{3 r^{k}}{4 \delta_0}  ( \def\avint{\mathop{\,\rlap{-}\!\!\int}\nolimits} \avint_{B_{\frac{3 r^k}{4}}} |f(y)|^q  
dy)^{1/q}}
\notag
\\
 &\leq \frac{4}{3} \delta_0.
\notag
\end{align}
Moreover,
\[
\omega(r^k) \leq \sum \omega(r^i) \leq C_0 \delta_0.
\]

 Therefore, $v$ satisfies an equation for which the conditions in Lemma \ref{app1} are satisfied. Consequently for a given $\tau>0$, we can find $\delta_0>0$ such that for some $w$ with universal $C^{2}$ bounds  we have that $||w-v||_{L^{\infty}(B_{1/2})} \leq \tau$. Now since $w$ has uniform $C^{2}$ bounds and $w(0)=0$,  there exists a  universal $C>0$ such that
\[
|w-Lx| \leq C|x|^2,
\]
where $L$ is the linear approximation for $w$ at $0$. We then choose $r$  small enough such that
\[
Cr^2= \frac{r^{1+\alpha}}{2},
\]
where $\alpha$ is as in \eqref{universal}. Subsequently, we  let $\tau= \frac{r^{1+\alpha}}{2}$ which decides the choice of $\delta_0$.  Then, by an application of triangle inequality we have,
\[
||v- L||_{L^{\infty}(B_r)} \leq r^{1+\alpha}.
\]

 Consequently by scaling back to $u$  we obtain
\begin{equation}\label{sc}
||u - \tilde  L_{k+1}||_{L^{\infty}(B_{r^{k+1}})} \leq  r^{k+1} r^{\alpha} \omega(r^k) \leq  r^{k+1} \omega(r^{k+1}),
\end{equation}
where  $\tilde L_{k+1}(x)\doteq \tilde L_k +  r^k \omega(r^k) L\bigg(\frac{x}{r^k}  \bigg)$. Note that in the last inequality in \eqref{sc} we also  used the following $\alpha-$decreasing property of $\omega$
\begin{equation}\label{dec}
r^{\alpha} \omega(r^k) \leq \omega(r^{k+1}),
\end{equation}
which is easily seen from the expression of $\omega$ as in \eqref{om} (see also the proof of  Lemma 4.7 in \cite{AB}).  This verifies the induction step. The conclusion now  follows by a standard analysis argument as in the proof of Lemma 4.9 in  \cite{AB}.

\end{proof}

The next result  is an  improvement of flatness result that allows to handle the case when the affine approximation have  small slopes  at a   \say{$k$th step}.  This corresponds to the degenerate alternative    in the iterative argument  in the proof of the main result Theorem \ref{main}.     
\begin{lem}\label{rt}
Let $u$ be a solution to
\begin{equation}\label{j1}
\bigg(\delta_{ij}+ (p-2) \frac{u_i u_j}{|\nabla u|^2} \bigg) u_{ij}=f\quad \text{in $B_1$},
\end{equation}
 with $|u| \leq 3$ and $u(0)=0$.  There exists  a universal $\ve_0>0$ such that if 
\begin{equation}\label{sml1}
\int_{0}^{1}  \left( \def\avint{\mathop{\,\rlap{-}\!\!\int}\nolimits} \avint_{B_s} |f|^q\right)^{1/q} ds \leq \ve_0,
\end{equation}
then there exists an affine function $L$,  with universal bounds,  and a universal $\eta \in (0,1)$  such that
\[
||u- L||_{L^{\infty}(B_{\eta})} \leq \delta_0 \eta^{1+\alpha}.
\]
Here $\delta_0$ is as in Lemma \ref{ap2} above. Without loss of generality we may take $\ve_0 < \delta_0^2$. 

\end{lem}

\begin{proof}

First note that  \eqref{sml1} implies that
\[
||f||_{L^q(B_{3/4})} < C\ve_0.
\]
We first show that given $\kappa>0$, there exists $\ve_0>0$ such that if  $u$ solves \eqref{j1} and $f$ satisfies the bound in \eqref{sml1}, then there exists a $p-$harmonic function $w$  such that
\begin{equation}\label{close}
||w-u||_{L^{\infty}(B_{1/2})} \leq  \kappa.
\end{equation}
Assume that \eqref{close} actually holds. It  then follows from the  $C^{1, \beta}$ regularity results for $p$-harmonic functions  in \cite{Db}, \cite{Le} and \cite{To} that there exists an affine function $L$ such that
\[
|w(x) -L(x)| \leq C|x|^{1+\beta}.
\]
We now choose  $\eta>0$ such that 
\[
C\eta^{1+\beta}= \frac{\delta_0}{2} \eta^{1+\alpha}\ \left(\text{This  crucially uses $\alpha < \beta$}\right).
\]
Subsequently, we choose $\kappa= \frac{\delta_0}{2} \eta^{1+\alpha}$, and this decides the choice of $\ve_0$. The conclusion of the lemma now  follows by an application of the triangle inequality. 

We are now going to prove \eqref{close}. Then there exists $\kappa_0>0$ and a sequence of pairs  $\{u_k, f_k\}$ which solves \eqref{j1} with $f_k$ satisfying \eqref{sml1} ( with $\ve_0=\frac{1}{k}$) such that $u_k$ is not $\kappa_0$ close to any such $w$. Then from uniform Krylov-Safanov type  H\"older estimates as in \cite{E} and Arzela-Ascoli, it follows that $u_k \to u_0$ uniformly in $B_{1/2}$ upto a subsequence.  We amke the following claim.
\medskip

 \emph{Claim: $u_0$ is $p-$harmonic.}
 \medskip
 
Once the claim is established, this would then be a contradiction for large enough $k$'s and thus \eqref{close} would follow. 

The proof is similar to that of the \emph{Claim} in  Lemma \ref{app1}. As before, we note that the stability  result in Theorem 3.8 in \cite{CCCS} cannot be directly applied because the operator $\Delta_{p}^N$ does not satisfy the structural assumptions in \cite{CCCS} because of singular dependence in the \say{gradient} variable.   We first observe  that it follows from  \cite{JLM} that in order to show that $u_0$ is $p-$ harmonic, it suffices to show that $u_0$ satisfies the viscosity formulation at points where the gradient of the test function does not vanish.

  Let $\phi$ be a $C^{2}$ test function which   strictly touches the graph of  $u$ from above at some point  $x_0 \in B_{1/2}$ such that  $\nabla \phi(x_0) \neq 0$.  We claim that 
\begin{equation}\label{cl0}
\Delta_{p}^N \phi (x_0) \geq 0.
\end{equation}
Suppose such is not the case. Then  there exists $\ve,r, \delta>0$ small enough such that
\begin{equation}\label{l1}
\begin{cases}
\Delta_{p}^N \phi(x) \leq - \ve\ \text{for $x \in B_r(x_0)$},
\\
\phi - u > \delta\ \text{on $\partial B_r(x_0)$}.
\end{cases}
\end{equation}
Moreover, we can also assume that in $B_r(x_0)$, we have that
\begin{equation}\label{ph}
|\nabla \phi | \geq \kappa>0.
\end{equation}
We now show that for every $k$, there exists   perturbed test functions $\phi + \phi_k$ with $\phi_k \in W^{2,q}$ such  that
\begin{equation}\label{con}
F^{*} ( \nabla (\phi +\phi_k), \nabla^{2} (\phi + \phi_k)) \leq f_k - \ve\ \text{in $B_r(x_0)$},\ \text{ with $F^{*}$ as in \eqref{up}}.
\end{equation}
Moreover, we can also ensure that  $ (\phi+ \phi_k) - u_k$ has a minimum  in $B_{r}(x_0)$ for large enough $k's$. This would  then contradict the viscosity formulation for $u_k$, and hence \eqref{cl0} would follow.   In an entirely analogous way, we will have that if  a $C^{2}$ test function strictly touches $u$ from below at $x_0$ then
\[
\Delta_{p}^N \phi (x_0) \leq 0,
\]
and consequently  we can assert  from the results in \cite{JLM} that $u_0$ is $p$-harmonic.  

Hence under the assumption that \eqref{l1} is valid, we now turn our attention to establish  \eqref{con}. We first observe that  because of \eqref{l1},  \eqref{ph},  the following inequality holds,
\begin{align}\label{s1}
& F^{*} ( \nabla (\phi +\phi_k), \nabla^{2} (\phi + \phi_k))  \leq \mathcal{P}^{+}_{\lambda, \Lambda} (\nabla^2 \phi_k)  + C(\kappa, ||\nabla^2 \phi||) |\nabla \phi_k|  -\ve,
\end{align}
with $\lambda, \Lambda$ as in \eqref{l}.
Here $\mathcal{P}_{\lambda, \Lambda}^{+}$ is the maximal Pucci operator defined as   in \eqref{max}. This inequality again follows by adding and subtracting $\Delta_{p}^N \phi$, by using \eqref{l1} and then by splitting considerations depending on whether
\[
|\nabla \phi_k| < \kappa/2\ \text{or}\ > \kappa/2.
\]
At this point, given $k$, we   look for $\phi_k$ which is a strong solution to
\begin{equation}\label{s}
\begin{cases}
 \mathcal{P}^{+}_{\lambda, \Lambda} (\nabla^2 \phi_k)  + C(\kappa, ||\nabla^2 \phi||) |\nabla \phi_k|= f_k\ \text{in $B_r(x_0)$},
\\
\phi_k= 0\ \text{on $\partial B_r(x_0)$}.
\end{cases}
\end{equation}
The existence of such strong solutions is again guaranteed by Corollary 3.10 in \cite{CCCS}. Moreover since $f_k \to 0$ in $L^{q}$, therefore  from the generalized maximum principle we have that 
\[
||\phi_k||_{L^{\infty}(B_r)} \to 0\ \text{as $k \to \infty$}.
\]
Now since $\phi - u$ has a strict  minimum at $x_0$, it follows that for large $k's$ that  $ (\phi+ \phi_k) - u_k$ would have a minimum in the inside of $B_r(x_0)$( since $\phi_k \equiv 0$ on $\partial B_r(x_0)$ and $\phi-u > \delta$ on $\partial B_r(x_0)$). However because  of \eqref{s1} and \eqref{s}  we also  have that  \eqref{con} holds which violates the viscosity formulation for $u_k$'s for large enough $k's$. Thus in view of our discussion above, we can assert that $u_0$ is $p-$harmonic  and  this concludes  the proof.

 \end{proof}

With this Lemma \ref{ap2} and Lemma \ref{rt} in hand, we now proceed with the proof of our main result.

\begin{proof}[Proof of Theorem \ref{main}]

 We will show that there exists an affine function $\tilde L$ such that
\begin{equation}\label{des}
|u(x)- \tilde L(x)| \leq C|x| K_0(4 |x|),
\end{equation}
 where $K_0(|x|)$ is defined as 
\[
K_0(|x|) \doteq \bigg(  \int_0^1  \left( \def\avint{\mathop{\,\rlap{-}\!\!\int}\nolimits} \avint_{B_s} |f|^q\right)^{1/q} ds  + 1 \bigg) |x| ^{\alpha/4} +    C_0(\alpha) \int_{0}^{|x|^{1/4}} \left( \def\avint{\mathop{\,\rlap{-}\!\!\int}\nolimits} \avint_{B_s} |f|^q\right)^{1/q} ds\],
and some universal $C$.

Likewise a similar affine approximation holds at all points in $B_{1/2}$ and consequently the estimates in \eqref{bm} follow by a standard real analysis argument. 

\medskip

We may assume that $u(0)=0$.  Now with  $\eta, \ve_0$ as in lemma \ref{rt} and  $\delta_0$ as in  lemma \ref{ap2},   assume the following hypothesis for a given $i \in \mathbb{N}$,
\begin{equation*}\label{H}[H]
\begin{cases}
\text{There exists affine function $L_i(x)\doteq <B_i, x>$ such that}\  ||u-L_i||_{L^{\infty}(B_{\eta^i})}  \leq \delta_0 \eta^i  \omega( \eta^i)
\\
\text{and}\ |B_i|  \leq 2  \omega(\eta^i).
\\
\end{cases}
\end{equation*}
Here $\omega$ is defined instead as 
\begin{equation}\label{om1}
\omega(\eta^k)\doteq \frac{1}{\ve_0} \sum_{i=0}^k \eta^{i \alpha} \omega_1(\eta^{k-i}),
\end{equation}
where we let $\omega_1$ to be
\[
\omega_1(t)\doteq  \max\left( \int_{0}^t\left( \def\avint{\mathop{\,\rlap{-}\!\!\int}\nolimits} \avint_{B_s
} |f|^q\right)^{1/q}ds,   t\right).
\]

By multiplying $u$ with a suitable constant we can assume that the Statement $[H]$ holds when $i=0$ with $L_0=0$.  Let $k$ be the first integer such that the Statement $[H]$  breaks. Then there  are two possibilities.

\emph{Case 1:} Suppose $k=\infty$. Then let given $x$, let $i\in \mathbb{N}$ be such that $|x| \sim \eta^{i}$. Then from the inequalities in $[H]$ and triangle inequality, it follows that
\begin{equation}\label{g1}
|u(x)| \leq |u(x) - <B_i, x>| +  |<B_i ,x>| \leq C_1 \eta^{i} \omega(\eta^i) \leq C_ |x| \omega( 2|x|) \leq  C |x| K_0( 4 |x|),
\end{equation}
and thus \eqref{des} follows with $\tilde L=0$. The last inequality in \eqref{g1} is seen as follows:
\begin{align}\label{b0}
 \omega(\eta^i)&= \frac{1}{\ve_0} \sum_{j=0}^i \eta^{j \alpha} \omega_1(\eta^{i-j})
\\
& \leq C \omega_1(\eta^{i/2}) \sum_{j=0}^{i/2}   \eta^{j\alpha} +  C \omega_1(1) \sum_{j=i/2}^{i} \eta^{j\alpha}\quad (\text{here we use $\omega_1$ is increasing})
\notag
\\
& \leq C \bigg(  \int_0^1  \left( \def\avint{\mathop{\,\rlap{-}\!\!\int}\nolimits} \avint_{B_s} |f|^q)^{1/q} ds + 1 \right) \eta ^{i\alpha/2} +    C_0(\alpha) \int_{0}^{\eta^{i/2}} \left( \def\avint{\mathop{\,\rlap{-}\!\!\int}\nolimits} \avint_{B_s} |f|^q\right)^{1/q} ds
\notag
\\
& \leq C K_0( 4 |x|)\quad (\text{using $|x| \sim \eta^{i}$}).
\notag
\end{align}

\emph{Case 2:} Suppose  instead that $k < \infty$. Then we have that the Statement $[H]$ is satisfied upto $k-1$. Now let
\[
v(x)\doteq \frac{u(\eta^{k-1}x)}{\eta^{k-1} \omega(\eta^{k-1})},
\]
which solves
\[
\bigg( \delta_{ij} +(p-2)\frac{v_i v_j}{|\nabla v|^2} \bigg) v_{ij}= \frac{\eta^{k-1} f(\eta^{k-1} x)}{\omega(\eta^{k-1})}.
\]
Moreover, from  the estimates in \eqref{H} for $i=k-1$ it follows that $|v| \leq 2+ \delta_0 \leq 3$.  Also by change of variable, we have that for $f_k(x)= \frac{\eta^{k-1} f(\eta^{k-1} x)}{\omega(\eta^{k-1})}$, the following holds,  
\begin{align}\label{d1}
& \int_{0}^{1}  \left( \def\avint{\mathop{\,\rlap{-}\!\!\int}\nolimits} \avint_{B_s} |f_k|^q\right)^{1/q} ds
\\
& \leq \ve_0  \frac{ \eta^{k-1} \int_{0}^{1}  \left( \def\avint{\mathop{\,\rlap{-}\!\!\int}\nolimits} \avint_{B_s} |f(\eta^{k-1} x)|^q dx\right)^{1/q} ds}{  \int_{0}^{\eta^{k-1}}  {\left( \def\avint{\mathop{\,\rlap{-}\!\!\int}\nolimits} \avint_{B_s} |f|^q dx\right)^{1/q}}}\notag
\\
& = \ve_0  \frac{ \int_{0}^{\eta^{k-1}}  ( \def\avint{\mathop{\,\rlap{-}\!\!\int}\nolimits} \avint_{B_s} |f|^q dx)^{1/q} ds}{  \int_{0}^{\eta^{k-1}}  {( \def\avint{\mathop{\,\rlap{-}\!\!\int}\nolimits} \avint_{B_s} |f|^q dx)^{1/q}} ds}\quad \left(\text{by change of variable}\right)
\notag
\\ 
&=\ve_0.
\notag
\end{align}
Here we have used also that
 \[
 \omega(\eta^{k-1}) \geq  \frac{1}{\ve_0} \int_{0}^{\eta^{k-1}}  \left( \def\avint{\mathop{\,\rlap{-}\!\!\int}\nolimits} \avint_{B_s} |f|^q dx\right)^{1/q}.
 \]

Hence, $v$  solves  an equation of the type \eqref{m} such that   the hypothesis in Lemma \ref{rt} is satisfied. Therefore, by applying Lemma \ref{rt}, we obtain that there exists an affine function $Lx= <\tilde A, x>$ such that
\[
||v- L||_{L^{\infty}(B_{\eta})}\leq  \delta_0 \eta^{1+\alpha}.
\]
Scaling back to $u$, we obtain with $L_{k} x\doteq <B_k, x>$, where $B_k\doteq\omega(\eta^{k-1})\tilde Ax$, that 
\begin{equation}\label{est2}
 ||u-L_k||_{L^{\infty}(B_{\lambda^k})}  \leq \delta_0 \eta^k \eta^{\alpha} \omega( \eta^{k-1}) \leq \delta_0 \eta^k \omega( \eta^k),
 \end{equation}
 where in the last inequality, we used the $\alpha-$decreasing property of $\omega$ (as  in \eqref{dec}). This property is easily seen from the expression of $\omega$  in \eqref{om1}.  However, since the Statement  $[H]$ does not hold for $i=k$,  we must  necessarily have
 \begin{equation}\label{f0}
 |B_k| \geq 2 \omega(\eta^k).
 \end{equation}
 We now let 
 \[
 \tilde v= \frac{u(\eta^k x)}{\eta^k  \omega(\eta^k)}.
 \]
 Then, we observe that  $\tilde v$ solves
 \[
 \bigg(\delta_{ij}+(p-2) \frac{\tilde v_i \tilde v_j}{|\nabla \tilde v|^2} \bigg)\tilde v_{ij} =\frac{\eta^k f(\eta^k x)}{\omega(\eta^k)}.
 \]
 Moreover, from \eqref{est2} we have, with 
 \begin{equation}\label{a}
 A= \frac{\omega(\eta^{k-1})\tilde A}{\omega(\eta^k)},
 \end{equation}
 that the following inequality holds
 \begin{equation}\label{est4}
 ||\tilde v - <A,x>||_{L^{\infty}(B_1)} \leq  \delta_0.
 \end{equation}
 Moreover, using that $|\tilde A| \leq C$, where $C$ is universal, and the $\alpha-$decreasing property of $\omega$, we obtain 
 \begin{equation}\label{est5}
 |A| = \frac{|\tilde A| \eta^{\alpha} \omega(\eta^{k-1})}{ \eta^\alpha \omega(\eta^{k})} \leq \frac{C}{\eta^{\alpha}}.
 \end{equation}
 Also \eqref{f0} implies
 \[
 |A| \geq 2.
 \]
 Now again by   change of variables it is seen that   $\tilde f_k$,  defined by
\begin{equation}\label{d10} 
\tilde f_k(x)\doteq  \frac{\eta^k f(\eta^k x)}{ \omega(\eta^k)},
\end{equation}
satisfies the estimate as in \eqref{d1}. 
Now using the fact that $\ve_0 < \delta_0^2$, we find that  $\tilde v$ satisfies the conditions in Lemma \ref{ap2}. Hence,  there exists an affine function $L_0 x\doteq  <A_0, x>$, with universal bounds depending on $\eta$( more specifically on $\frac{C}{\eta^{\alpha}}$), such that
\begin{equation}\label{tv}
|\tilde v(x)- L_0(x)| \leq C|x|  K_{\tilde f_k}(|x|),\quad |x|<1,
\end{equation}
where $ K_{\tilde f_k}(|x|)= |x|^{\alpha/2} + \int_{0}^{|x|^{1/2}} ( \def\avint{\mathop{\,\rlap{-}\!\!\int}\nolimits} \avint_{B_s} |\tilde f_k|^q)^{1/q} ds$ with $\tilde f_k$ as in \eqref{d10}.  Then, by scaling back to $u$, letting $\eta^{k} x$ as our new $x$,  we obtain for $|x| \leq \eta^k$ that the following holds by change of variables,
\begin{align}\label{b1}
& |u(x)- \omega(\eta^k) <A_0, x>| \leq C |x| \left(  \omega(\eta^k) |y|^{\alpha/2} +  \int_{0}^{\eta^{k} |y|^{1/2}} ( \def\avint{\mathop{\,\rlap{-}\!\!\int}\nolimits} \avint_{B_s} |f|^q)^{1/q} ds\right)\  \left(\text{ $y=\eta^{-k} x$}\right)
\\
& \leq  C |x| \left(  \omega(\eta^{k/2}) |y|^{\alpha/2} +  \int_{0}^{\eta^{k/2} |y|^{1/2}} ( \def\avint{\mathop{\,\rlap{-}\!\!\int}\nolimits} \avint_{B_s} |f|^q)^{1/q} ds\right)\ \left(\text{using $\eta^k \leq \eta^{k/2}$ and $\omega(\eta^{k}) \leq \omega(\eta^{k/2})$}\right)
\notag
\\
& = C |x| \left(  \omega(\eta^{k/2}) |y|^{\alpha/2} +  \int_{0}^{ |x|^{1/2}} ( \def\avint{\mathop{\,\rlap{-}\!\!\int}\nolimits} \avint_{B_s} |f|^q)^{1/q} ds\right).
\notag
\end{align}
Now, let $j$ be the smallest integer such that  $|y| \leq  \eta^{j}$. Then, we have that
\begin{align}\label{b8}
  \omega(\eta^{k/2}) |y|^{\alpha/2}& \leq \omega(\eta^{k/2}) \eta^{j\alpha/2}
\\
& =\frac{1}{\ve_0}  \sum_{i=j/2}^{\frac{k+j}{2}} \eta^{ i\alpha} \omega_1(\eta^{\frac{k+j}{2} - i}) \leq \omega(\eta^{\frac{k+j}{2}})
\notag
\\
& \leq C\bigg[ \bigg(\int_{0}^{ 1} \left( \def\avint{\mathop{\,\rlap{-}\!\!\int}\nolimits} \avint_{B_s} |f|^q)^{1/q} \bigg) |x|^{\alpha/4} +  \int_{0}^{ |x|^{1/4}} ( \def\avint{\mathop{\,\rlap{-}\!\!\int}\nolimits} \avint_{B_s} |f|^q\right)^{1/q} ds \bigg] 
\notag
\\
&\leq C K_0(4|x|)\quad (\text{using $y=\eta^{-k} x$}),
\notag
\end{align}
where the last inequality in \eqref{b8} follows from a computation as in \eqref{b0}. This implies that \eqref{des} holds with $\tilde L x\doteq <\omega(\eta^k) A_0, x>$, when $|x| \leq \eta^k$. 

Now when $|x| \geq \eta^k$,   one can show  that 
\begin{equation}\label{cl1}
|u(x) | \leq C |x| \omega(2|x|) \leq C |x|  K_0(4 |x|).
\end{equation}
This follows from the fact that with $L_ix\doteq <B_i,x>$ we have for $i=0,\ldots, k-1$,
\[
||u-L_i||_{L^{\infty}(B_{\eta^i})}  \leq \delta_0 \eta^i \omega( \eta^i) 
\]
and
\[
| B_i| \leq 2 \omega(\eta^i)
\]
because \eqref{H} holds upto $k-1$. And moreover for $i=k$, we  again   have
\[
||u-L_k||_{L^{\infty}(B_{\eta^k})}  \leq \delta_0 \eta^k \omega( \eta^k).
\]
In this case, instead the following bound holds
\[
|B_k| \leq C \omega(\eta^{k-1}) \leq \frac{C \omega(\eta^k)}{\eta^\alpha}\ \text{using $\alpha-$decreasing property of $\omega$}
\]
Using such estimates, it is easy to see that   \eqref{cl1}  holds. Now note that with   \linebreak $\tilde L x\doteq < \tilde B, x>$, with $B\doteq \omega(\eta^k) A_0$,  we also  have  the following bound
\begin{equation}\label{ai}
 |\tilde B| \leq C \omega(\eta^k).
 \end{equation}
Therefore, it follows from \eqref{cl1} and the estimate \eqref{ai} above that 
  \begin{equation}
  |u(x)- \tilde L(x) | \leq C |x| K_0(4|x|)
  \end{equation}
  also holds when $|x| \geq \eta^k$, for a possibly different $C$.  Hence  the estimate in \eqref{des} follows with $\tilde L x\doteq<\tilde B, x>$  and this finishes the proof of the theorem.

\end{proof}

\subsection{Proof of Theorem \ref{main1}}
In this subsection, we assume that $u$ is a $W^{2,m}$ viscosity solution to
\begin{equation}\label{m2}
\bigg(\delta_{ij} + (p-2) \frac{u_i u_j}{|\nabla u|^2} \bigg) u_{ij} =f,
\end{equation}
where $f \in L^{m}$ for some $m>n$.  We now state and prove the counterparts of the approximation lemmas in this situation.  The analogue of Lemma \ref{app1} is as follows.

\begin{lem}\label{ap01}
Let $u$ be a  $W^{2, m}$ viscosity solution to
\begin{equation}
\bigg(\delta_{ij} + (p-2) \frac{(\delta u_i+A_i)(\delta u_j +A_j)}{ |\delta \nabla u+A|^2} \bigg) u_{ij}=f\quad \text{in $B_1$},
\end{equation}
 with $|u| \leq 1$ and $|A| \geq 1$.  Given $\tau>0$, there exists  $\delta_0=\delta_0(\tau)>0$  such that if  
\[
\delta, \left(\frac{1}{|B_{3/4}|} \int_{B_{3/4}}  |f|^m \right)^{1/m} \leq \delta_0,
\]
then $||w-u||_{L^{\infty}(B_{1/2})} \leq \tau$ for some $w \in C^{2}(\overline{B_{1/2}})$ satisfying $w(0)=0$  with universal $C^{2}$ bounds depending only on $n, p$ and independent of $|A|$. 
\end{lem}

\begin{proof}
The proof is identical to that of Lemma \ref{app1} and so we omit the details.

\end{proof}

We now state the counterpart of  Lemma \ref{ap2}.

\begin{lem}\label{ap02}
Let $u$ be  a viscosity  solution to
\[
\bigg(\delta_{ij}+ (p-2) \frac{u_i u_j}{|\nabla u|^2} \bigg) u_{ij}=f
\]
in $B_1$ with  $u(0)=0$. Then there exists a universal $\delta_0>0$, such that if for some  $A \in \R^n$ satisfying  $M \geq|A| \geq 2$  we have 
\[
||u- <A,x>||_{L^{\infty}(B_1)} \leq \delta_0,
\]
and also 
\[
||f||_{L^m(B_1)}  \leq \delta_0^2,
\] 
then there exists an affine function  $ L_0$ such that
\begin{equation}\label{difr}
\begin{cases}
1 \leq|\nabla L_0|\leq M+1
\\
|u(x)-  L_0(x) | \leq C|x|^{1+\alpha_0}
\end{cases}
\end{equation}
where $\alpha_0 < \text{min}( \alpha, 1-n/m)$.  Moreover, $\delta_0$ can be chosen independent of $M$. 
\end{lem}

\begin{proof}
As in the proof of Lemma \ref{ap2}, we show that  for every $k \in \mathbb{N}$, there exists affine functions $\tilde L_kx = <A_k ,x>$ such that
\begin{equation}\label{itt}
\begin{cases}
||u- \tilde L_k x||_{L^{\infty}(B_{r^k})}\leq \delta_0 r^{k(1+\alpha_0)},
\\
|A_k - A_{k+1}| \leq C \delta_0 r^{k\alpha_0},
\end{cases}
\end{equation}
for some $r<1$ universal independent of $\delta_0$. The conclusion of the lemma then follows from \eqref{itt} in a standard way. 
We first observe that \eqref{itt} holds for $k=0$ with $A_0=A$. Moreover the non-degeneracy condition as in  \eqref{ndeg} is easily verified in this situation provided $\delta_0$ is small enough. Now assume \eqref{itt} holds upto some $k$.  We then define 
\[
v= \frac{u-\tilde L_k (r^k x)}{\delta_0r^{k(1+\alpha_0)}}.
\]
Then $v$ solves in $B_1$
\[
\bigg(\delta_{ij} + (p-2) \frac{ (\delta_0 r^{k \alpha_0} v_i + (A_k)_i) (\delta_0 r^{k \alpha_0} v_j + (A_k)_j)}{| \delta_0 r^{k\alpha_0} \nabla v + A_k|^2} \bigg) v_{ij}= f_k,
\]
where $f_k$ is defined as
\[
f_k(x)= r^{k(1-\alpha_0)} \frac{f(r^k x)}{\delta_0}.
\]
Now by change of variable it is seen that
\[
||f_k||_{L^m(B_1)} = r^{k(1-n/m) - \alpha_0} \frac{1}{\delta_0} ||f||_{L^m(B_{r^k})} \leq \delta_0.
\]
Note that over here, we crucially used the hypothesis of the lemma i.e,
\[
||f||_{L^m(B_1)} \leq \delta_0^2,
\]
and the fact that  $\alpha_0 < 1-n/m$.  Therefore, $v$ satisfies the hypothesis of Lemma \ref{ap01} and at this point  we can repeat the arguments in the proof of Lemma \ref{ap2} to conclude that there exists $ \tilde L_{k+1}(x) = \tilde L_k (x) + \delta_0 r^{k(1+\alpha_0)}L( \frac{x}{r^k})$, where $L$ has universal bounds such that \eqref{itt} holds for $k+1$. This verifies the induction step and the conclusion of the lemma thus follows.

\end{proof}

We also have the following lemma which is the analogue of Lemma \ref{rt}.

\begin{lem}\label{rt1}
Let $u$ be a solution of
\begin{equation}\label{j10}
\bigg(\delta_{ij}+ (p-2) \frac{u_i u_j}{|\nabla u|^2} \bigg) u_{ij}=f \quad \text{in $B_1$},
\end{equation}
 with $|u| \leq 3$ and $u(0)=0$.  There exists  a universal $\ve_0>0$ such that if 
\begin{equation}\label{sml5}
||f||_{L^m(B_1)} \leq \ve_0,
\end{equation}
then there exists an affine function $L$  with universal bounds  and a universal $\eta \in (0,1)$  such that
\[
||u- L||_{L^{\infty}(B_{\eta})} \leq \delta_0 \eta^{1+\alpha_0},
\]
where $\delta_0$ is as in Lemma \ref{ap02} above. Without loss of generality we may take $\ve_0 < \delta_0^2$. 

\end{lem}

\begin{proof}
The proof is again identical to that of Lemma \ref{rt} and thus we skip the details. 

\end{proof}

With Lemmas \ref{ap01}--\ref{rt1} in hand, we now proceed with the proof of Theorem \ref{main1}. 

\begin{proof}[ Proof of Theorem \ref{main1}]
It suffices to show that at $0$, there exists an affine function $\tilde L$  with universal bounds such that 
\begin{equation}\label{des1}
|u(x)- \tilde L(x)| \leq C|x|^{1+\alpha_0}.
\end{equation}
We also assume that $u(0)=0$.  Now with  $\eta, \ve_0$ as in Lemma \ref{rt1} and  $\delta_0$ as in  Lemma \ref{ap02},   assume the following hypothesis for a given $i \in \mathbb{N}$,
\begin{equation}\label{h1}[H1]
\begin{cases}
\text{There exists affine function $L_i(x)\doteq <B_i, x>$ such that}\  ||u-L_i||_{L^{\infty}(B_{\eta^i})}  \leq \frac{\delta_0}{\ve_0} \eta^{i (1+\alpha_0)}
\\
\text{and}\ |B_i|  \leq  \frac{2}{\ve_0} \eta^{i\alpha_0}.
\end{cases}
\end{equation}
By multiplying $u$ with a suitable constant, we may assume that the hypothesis  holds for $i=0$ with $L_0=0$.  We can also assume that 
\begin{equation}\label{nor}
||f||_{L^m(B_1)} \leq 1.
\end{equation}

Let $k$ be the smallest integer such that \eqref{h1} fails. Then as in the proof of Theorem \ref{main}, there are two possibilities.

\medskip

\emph{Case 1:} Suppose $k=\infty$. Then in this case, \eqref{des1} is seen to hold with $\tilde L=0$. 

\medskip

\emph{Case 2:} Suppose instead that $k<\infty$. Then we have that the hypothesis is satisfied upto $k-1$.  As before, we let
\[
v(x)= \ve_0 \frac{u(\eta^{k-1}x)}{\eta^{(k-1)(1+\alpha_0)}},
\]
which solves in $B_1$
\[
\bigg( \delta_{ij}+ (p-2) \frac{v_i v_j}{|\nabla v|^2} \bigg) v_{ij} = f_k,
\]
where 
\[
f_k(x) = \ve_0 \eta^{(k-1)(1-\alpha_0)} f(\eta^{k-1} x).
\]
Then  by change of variable and \eqref{nor}, it  is again seen that  $||f_k||_{L^m} \leq \ve_0 $. Moreover, from \eqref{h1} and triangle inequality it follows that $|v| \leq 2 + \delta_0 \leq 3$.  Thus the hypothesis of Lemma \ref{rt1} is satisfied and consequently there exists $L x= <\tilde A, x>$ affine such that
\[
||v- L||_{L^{\infty}(B_\eta)} \leq \delta_0 \eta^{1+\alpha_0}.
\]
By scaling back to $u$, we obtain with $L_k x\doteq  B_k x$, with $B_k= \frac{\eta^{(k-1)\alpha_0}}{\ve_0} \tilde A $, that the following holds,
\[
||u-L_k||_{L^{\infty}(B_{\eta^k})} \leq \frac{\delta_0}{\ve_0} \eta^{k(1+\alpha_0)}.
\]
However since Statement $[H1]$ fails,  we  must necessarily have
\[
|B_k| \geq \frac{2}{\ve_0} \eta^{k\alpha_0}.
\]
If we now let
\[
\tilde v(x)= \ve_0 \frac{u(\eta^k x)}{\eta^{k(1+\alpha_0)}},
\]
then, as in the proof of Theorem \ref{main}, it can be easily checked  that $\tilde v$ solves an equation of the type \eqref{m} such that the hypothesis of Lemma \ref{ap02} is verified. Hence there exists an affine function $L_0 x= <A_0, x>$,  with universal bounds depending on $\eta$, such that
\[
|\tilde v- L_0 x| \leq C|x|^{1+\alpha_0}.
\]
By scaling back to $u$ we obtain  that, with $\tilde L(x)\doteq \frac{\eta^{k\alpha_0}}{\ve_0}<  A_0, x>$, the following estimate holds for $|x| \leq \eta^k$, 
\begin{equation}\label{g10}
|u(x)- \tilde L(x)| \leq C|x|^{1+\alpha_0}.
\end{equation}
The rest of the argument is again the same  as in  the proof of Theorem \ref{main}, which allows us to conclude that the estimate \eqref{g10} holds also when $|x| \geq \eta^k$. This finishes the proof of the theorem.

\end{proof}

  \end{document}